\documentclass[11pt,reqno]{amsart}
\usepackage[T1]{fontenc}
\usepackage[utf8]{inputenc}
\usepackage{lmodern}
\usepackage{amsmath,amssymb,amsthm,mathtools}
\usepackage{microtype}
\usepackage{enumitem}
\usepackage[margin=1.1in]{geometry}
\usepackage{hyperref}
\usepackage[capitalize,nameinlink]{cleveref}
\hypersetup{colorlinks=true,linkcolor=blue,citecolor=blue,urlcolor=blue,
 pdftitle={Weighted regularity and plurithinness in Cn},
 pdfauthor={Nurbek Kh. Narzillaev}}
\setlist[enumerate]{label=\textup{(\roman*)},leftmargin=*}
\newtheorem{theorem}{Theorem}[section]
\newtheorem{proposition}[theorem]{Proposition}
\newtheorem{lemma}[theorem]{Lemma}
\newcommand{\C}{\mathbb C}
\newcommand{\R}{\mathbb R}
\newcommand{\D}{\mathbb D}
\newcommand{\PSH}{\operatorname{PSH}}
\newcommand{\rank}{\operatorname{rank}}
\newcommand{\calL}{\mathcal L}
\newcommand{\cl}{\overline}
\newcommand{\reg}{\mathrm{reg}}
\newcommand{\ddc}{dd^c}
\newcommand{\lp}{\log^+}
\newcommand{\calF}{\mathcal F}
\newcommand{\Iso}{\operatorname{Iso}}
\newcommand{\norm}[1]{\lVert #1\rVert}

\title[Weighted regularity and plurithinness]{Weighted regularity and plurithinness in $\C^n$}
\author{Nurbek Kh. Narzillaev}
\address{National University of Uzbekistan, Tashkent, Uzbekistan}
\email{n.narzillaev@nuu.uz}
\date{}

\begin{document}
\begin{abstract}
We study weighted pluricomplex extremal functions with a fixed weight and a variable logarithmic growth parameter. We show that the regular parameters at a point form the empty set, the smallest contact parameter alone, or the whole contact ray. For continuous weights, we represent the right parameter derivative by probability measures satisfying quasi-everywhere inequalities. This gives a criterion for equality with the unweighted Green function. We construct two counterexamples to preservation of weighted regularity, using smooth plurisubharmonic weights and polynomially convex compact sets that are non-pluripolar and non-plurithin at every point. For the product example, we compute the derivative on the entire exceptional polydisc through planar Green functions and determine exactly where equality holds. The second example is globally pluriregular but has a positive right derivative. We also characterize the points where local pluriregularity fails without plurithinness, using negligible subsets and the plurifine topology, and prove preservation of regularity under local strict plurisubharmonicity of the extending weight.
\end{abstract}
\maketitle

\section{Introduction and main results}

We study how weighted regularity changes when the logarithmic growth parameter increases. This section introduces the parameter sets and states the main results: their possible forms, two counterexamples to preservation of regularity, and a characterization of non-plurithin points that are not locally pluriregular.

Let $K\subset\C^n$ be a non-pluripolar compact set, and let $\psi$ be a bounded real-valued function on $K$. For $\delta>0$, define
\begin{equation}\label{eq:intro-envelope}
 V_{\delta,K,\psi}(z)
 =\sup\{u(z):u\in\calL_\delta,\ u\le\psi\text{ on }K\},
 \qquad F_\delta=V_{\delta,K,\psi}^*.
\end{equation}
Here $\calL_\delta$ is the class of plurisubharmonic functions on $\C^n$ with logarithmic growth at most $\delta$, and the star denotes upper semicontinuous regularization. The precise definitions are recalled in \cref{sec:preliminaries}. For $\delta=1$ and $\psi=0$, we obtain the classical Siciak--Zahariuta extremal function \cite{Siciak1981}
\[
 V_K:=V_{1,K,0}.
\]
Its regularization $V_K^*$ is the pluricomplex Green function with pole at infinity. We keep $K$ and $\psi$ fixed throughout each parameter comparison.

Weighted extremal functions arise in polynomial approximation and pluripotential theory; see \cite{BloomLevenberg2003,Klimek1991,Levenberg2017}. We use the notion of weighted regularity studied by Sadullaev \cite{Sadullaev2016}, Alan \cite{Alan2019}, and Narzillaev \cite{Narzillaev2021}: at a contact parameter, $K$ is $(\delta,\psi)$-regular at $a\in K$ if
\[
 F_\delta(a)=\psi(a).
\]
This is an equality with the weight. Continuity of the weighted extremal function alone does not imply it; see \cite[Example~2.3]{Alan2019}.

Alan \cite[Open Problem~3.7]{Alan2019} asked whether $(\delta_1,\psi)$-regularity at a point implies $(\delta_2,\psi)$-regularity when $0<\delta_1<\delta_2$ and $\psi$ extends to a function in $\calL_{\delta_1}^+$. Alan uses $\alpha$ for the growth parameter; here we use $\delta$, as in \cite{Narzillaev2021}.

Narzillaev introduced the contact parameter set and the global regular parameter set. In our notation they are
\begin{align*}
 \Lambda(K,\psi)&=\{\delta>0:V_{\delta,K,\psi}|_K=\psi\},\\
 \Lambda_{\reg}(K,\psi)&=\{\delta\in\Lambda(K,\psi):F_\delta|_K=\psi\}.
\end{align*}
If $\Lambda(K,\psi)$ is nonempty, it is either $(0,\infty)$ or $[\lambda,\infty)$ for some $\lambda>0$; see \cite[Section~4, Proposition~1]{Narzillaev2021}. We also use the pointwise set
\[
 \Lambda_{\reg}(a;K,\psi)
 =\{\delta\in\Lambda(K,\psi):F_\delta(a)=\psi(a)\},\qquad a\in K.
\]
This pointwise notation is introduced here. The global set is its intersection over $a\in K$.

The interval descriptions in \cite[Section~4]{Narzillaev2021} allow nondegenerate bounded intervals of regular parameters. Our first theorem excludes this possibility.

\begin{theorem}\label{thm:main-rigidity}
Let $K$ be a non-pluripolar compact set, let $\psi$ be bounded, and suppose that $\Lambda=\Lambda(K,\psi)$ is nonempty.
\begin{enumerate}
\item If $\Lambda=[\lambda,\infty)$ with $\lambda>0$, then
\[
 \Lambda_{\reg}(a;K,\psi)
 \in\{\varnothing,\{\lambda\},[\lambda,\infty)\}
 \quad\text{for every }a\in K.
\]
The same alternatives hold for $\Lambda_{\reg}(K,\psi)$.
\item If $\Lambda=(0,\infty)$, then each of these regular parameter sets is either empty or equal to $(0,\infty)$.
\end{enumerate}
\end{theorem}

For fixed $a$, the function $\delta\mapsto F_\delta(a)-\psi(a)$ is nonnegative on $\Lambda$, concave, and nondecreasing. These three properties prove the theorem. In particular, if regularity holds at $\delta_1$ but fails at a larger parameter under the assumptions of Alan's question, then
\[
 \Lambda(K,\psi)=[\delta_1,\infty),\qquad
 \Lambda_{\reg}(a;K,\psi)=\{\delta_1\}.
\]
Thus loss of regularity is possible only when the starting parameter is the smallest contact parameter. The following construction realizes this case.

We write $\D=\{\zeta\in\C:|\zeta|<1\}$ and $\cl\D=\{\zeta\in\C:|\zeta|\le1\}$ for the open and closed unit disks. The notation $\cl\D^{\,n-1}$ means the Cartesian product of $n-1$ closed unit disks.

\begin{theorem}\label{thm:main-example}
For every integer $n\ge2$, there is a polynomially convex compact set $K_n\subset\C^n$ that is non-pluripolar and non-plurithin at every point and has the following property. For every $\tau>0$, there exists
\[
 \Psi_\tau\in C^\infty(\C^n)\cap\calL_\tau^+,
 \qquad \psi_\tau=\Psi_\tau|_{K_n},
\]
such that
\[
 \Lambda(K_n,\psi_\tau)=[\tau,\infty),\qquad
 \Lambda_{\reg}(K_n,\psi_\tau)=\{\tau\}.
\]
More precisely, put $Z_n=\{2\}\times\cl\D^{\,n-1}$. The set of locally pluriregular points of $K_n$ is $K_n\setminus Z_n$, and
\begin{equation}\label{eq:main-gap}
 F_\delta(a)-\psi_\tau(a)
 \ge\frac49(\delta-\tau)\log2>0
 \quad(a\in Z_n,\ \delta>\tau).
\end{equation}
The nonnegative $(1,1)$-form $\ddc\Psi_\tau$ has rank $n-1$ near $Z_n$.
\end{theorem}

This gives a negative answer to Alan's Open Problem~3.7, even under pointwise non-pluripolarity, non-plurithinness, polynomial convexity, and smoothness of the extending weight. The product geometry follows Alan's Example~3.16. Here we give an explicit weight, choose its regular contact parameter, and prove the estimate \eqref{eq:main-gap}.

The unweighted Green function is positive on the exceptional set in \cref{thm:main-example}. One might therefore ask whether global pluriregularity is enough to preserve weighted regularity. The next theorem shows that it is not.

\begin{theorem}\label{thm:global-example}
For every $n\ge2$, there is a polynomially convex compact set $K\subset\C^n$ that is non-pluripolar and non-plurithin at every point and satisfies
\[
 V_K^*|_K=0.
\]
For every $\tau>0$, there is a function $\Psi_\tau\in C^\infty(\C^n)\cap\calL_\tau^+$ such that, with $\psi_\tau=\Psi_\tau|_K$,
\[
 \Lambda(K,\psi_\tau)=[\tau,\infty),\qquad
 \Lambda_{\reg}(K,\psi_\tau)=\{\tau\}.
\]
At a fixed point $a\in K$, regularity fails for every $\delta>\tau$, and
\[
 D_\tau(a)\ge\frac{8\log3-4\log2}{9}>0=V_K^*(a).
\]
Here $D_\tau(a)$ is the right derivative of $\delta\mapsto F_\delta(a)$ at $\tau$, defined in \eqref{eq:sensitivity}.
\end{theorem}

The compact set combines a variant of Sadullaev's chair construction \cite[pp.~150--151]{Sadullaev2016} with a ball and a sequence of small product sets. The ball determines the initial weighted envelope. The small product sets ensure non-pluripolarity at every point. In two variables we obtain an exact formula for $F_{\tau+h}(a)$ on a nontrivial interval of $h$ and compute $D_\tau(a)$; see \cref{prop:global-bounds}.

We next consider Alan's Open Problem~3.17, which asks for a characterization of the remaining set in the decomposition
\begin{equation}\label{eq:RTS}
 \begin{aligned}
 R&=\{a\in K:K\text{ is locally pluriregular at }a\},\\
 T&=\{a\in K:K\text{ is plurithin at }a\},\\
 S&=K\setminus(R\cup T).
 \end{aligned}
\end{equation}
Here $K$ is assumed to be non-pluripolar at every point. Local pluriregularity implies non-plurithinness, so the three sets are disjoint. A negligible set is equivalently a pluripolar set; we recall this classical fact in \cref{sec:preliminaries}.

\begin{theorem}\label{thm:main-S}
Let $K\subset\C^n$ be compact and non-pluripolar at every point. A point $a\in K$ belongs to $S$ if and only if there is a negligible subset $N\subset K$ such that $a\in N$, $N$ is non-plurithin at $a$, and $K\setminus N$ is plurithin at $a$.
\end{theorem}

The theorem gives a characterization through negligible subsets and plurithinness, without a local extremal function in its statement. It follows from classical negligibility and pluripolar removal results. We also identify $K\setminus R$ as the largest relatively plurifine open negligible subset of $K$. Its plurifine limit points that belong to the subset are exactly the points of $S$; see \cref{prop:canonical-S}.

Two further results explain the role of the weight. If an admissible extension is strictly plurisubharmonic near $a$, then weighted regularity at $a$ is equivalent to local pluriregularity; see \cref{thm:local-strict}. At a regular parameter $\tau$, the right derivative is finite and satisfies $D_\tau(a)\ge V_K^*(a)$. Equality holds exactly when $F_\delta(a)$ is affine in $\delta$ for every $\delta\ge\tau$. For continuous weights, \cref{thm:derivative} gives a criterion for this equality using only $K$, the weight, and unweighted plurisubharmonic tests. For the product example, \cref{thm:product-derivative} computes the derivative at every point of the exceptional polydisc and shows that equality holds exactly on a specified smaller polydisc. The second example gives strict inequality even though $V_K^*(a)=0$.

Section~\ref{sec:preliminaries} recalls the required facts. Section~\ref{sec:parameters} proves \cref{thm:main-rigidity} and studies the right derivative. Sections~\ref{sec:local} and~\ref{sec:fine} treat local regularity and prove \cref{thm:main-S}. The two examples are constructed in Sections~\ref{sec:construction} and~\ref{sec:global-example}.

\section{Preliminaries}\label{sec:preliminaries}

We collect the notation and standard facts used in the proofs.

We use the normalization of \cite{Sadullaev2016}:
\[
 d^c=\frac{\partial-\bar\partial}{4i},\qquad
 \ddc=\frac{i}{2}\partial\bar\partial.
\]
For a real-valued function $u\in C^2(\Omega)$, where $\Omega\subset\C^n$ is open,
\[
 \ddc u=\frac{i}{2}\sum_{j,k=1}^n
 \frac{\partial^2u}{\partial z_j\partial\bar z_k}\,
 dz_j\wedge d\bar z_k.
\]
Thus $u$ is plurisubharmonic if and only if $\ddc u\ge0$, and strictly plurisubharmonic if and only if $\ddc u>0$, where positivity refers to $(1,1)$-forms. Write $\omega_0=\ddc\norm z^2$ for the standard positive form. For continuous functions, inequalities involving $\ddc$ are understood in the sense of currents. In particular, $\ddc\Psi\ge c\omega_0$ on $B(a,R)$ means that $\Psi-c\norm{z-a}^2$ is plurisubharmonic there.

We write $\lp t=\max\{\log t,0\}$, with $\lp0=0$, and use the Euclidean norm on $\C^n$. The open Euclidean ball with center $a$ and radius $r$ is denoted by $B(a,r)$; in one complex variable we also write $D(a,r)$ for this disk. For $\delta>0$, define
\[
 \calL_\delta=\{u\in\PSH(\C^n):u(z)\le\delta\lp\norm z+C_u\},
 \qquad
 \calL_\delta^+=\{u\in\calL_\delta:u(z)\ge\delta\lp\norm z-C_u\}.
\]
The inequalities are required on all of $\C^n$. We allow different constants in the upper and lower bounds. In particular, $\calL_\delta=\delta\calL_1$.

Throughout the paper $K$ is non-pluripolar and compact, and weights are bounded and real-valued, unless additional hypotheses are stated. For a locally bounded above function $f$, its upper semicontinuous regularization is $f^*(z)=\limsup_{\zeta\to z}f(\zeta)$. The envelope in \eqref{eq:intro-envelope} has a finite regularization in $\calL_\delta^+$ and agrees with it outside a pluripolar set. We use ``quasi-everywhere'' to mean outside a pluripolar set.

A set is pluripolar if, near each point, it is contained in the $-\infty$ locus of a plurisubharmonic function that is not identically $-\infty$. A set is negligible if it is locally contained in a set $\{u<u^*\}$, where $u$ is the supremum of a locally uniformly bounded above family of plurisubharmonic functions. Negligible sets and pluripolar sets coincide; countable unions of such sets are again negligible. These are the classical negligibility results of Bedford and Taylor; see \cite{BedfordTaylor1982} and \cite[Theorem~4.7.6 and Corollaries~4.7.9--4.7.10]{Klimek1991}.

For a bounded set $E$, not necessarily compact, $V_E$ is defined by the same unweighted supremum as above. If $P$ is pluripolar and $E\cup P$ is bounded, then
\begin{equation}\label{eq:polar-removal}
 V_{E\cup P}^*=V_E^*.
\end{equation}
We use this invariance only for regularized extremal functions; see \cite[Theorem~5.2.4 and Corollary~5.2.5]{Klimek1991}.

Put
\[
 m=\inf_K\psi,\quad M=\sup_K\psi.
\]
Directly from the definitions,
\begin{gather}
 V_{\delta,K,\psi}=\delta V_{1,K,\psi/\delta},\qquad
 V_{\delta,K,\psi+c}=V_{\delta,K,\psi}+c,\notag
\\
 \delta V_K+m\le V_{\delta,K,\psi}\le\delta V_K+M,\notag
\\
 \delta V_K^*+m\le F_\delta\le\delta V_K^*+M.\label{eq:sandwich}
\end{gather}
The star in \eqref{eq:sandwich} is essential: $V_K=0$ on $K$, whereas $V_K^*$ can be positive at points of $K$.

We use one elementary consequence of negligibility. If $U$ and $W$ are locally bounded above envelopes of plurisubharmonic functions with finite regularizations, then, for $s,t\ge0$,
\begin{equation}\label{eq:regularized-sum}
 (sU+tW)^*=sU^*+tW^*.
\end{equation}
Indeed, outside a pluripolar set the unregularized sum agrees with the plurisubharmonic function $sU^*+tW^*$. Such a function is recovered by upper regularization of its values off that set, which proves \eqref{eq:regularized-sum}. Terms with zero coefficient are omitted.

For $a\in K$, local pluriregularity means
\[
 V_{K\cap\cl B(a,r)}^*(a)=0\quad\text{for every }r>0.
\]
Local $(\delta,\psi)$-regularity is defined in the same way using the weighted envelope and the value $\psi(a)$. We restrict the obstacle to each smaller compact without changing notation. When $\delta\in\Lambda(K,\psi)$, every compact subset $L\subset K$ also has contact at that parameter:
\begin{equation}\label{eq:inherit-contact}
 V_{\delta,L,\psi}=\psi\quad\text{on }L.
\end{equation}
Indeed, on $L$ the envelope lies between $V_{\delta,K,\psi}=\psi$ and its obstacle.

For $E\Subset B$, with $B$ a ball, the relative extremal function is
\begin{equation}\label{eq:relative}
 h_{E,B}=\sup\{u\in\PSH(B):-1\le u\le0,\ u|_E\le-1\}.
\end{equation}
The lower bound $-1$ does not change its regularization. We use the standard equivalence between local pluriregularity at $a$ and the equalities
\[
 h_{K\cap\cl B(a,r),B(a,R)}^*(a)=-1\qquad(0<r<R)
\]
for sufficiently small concentric balls. This is the local relative-extremal characterization of regularity; see \cite{Klimek1991,Levenberg2017}.

We say that $K$ is non-pluripolar at every point if $K\cap B(a,r)$ is non-pluripolar for every $a\in K$ and $r>0$. This stronger local assumption will be stated whenever it is needed.

A set $E\subset\C^n$ is plurithin at $a$ if $a$ is not an accumulation point of $E\setminus\{a\}$, or if a plurisubharmonic function $u$ near $a$ satisfies
\[
 \limsup_{E\setminus\{a\}\ni z\to a}u(z)<u(a).
\]
The puncture is part of the definition. Finite unions of sets plurithin at $a$ are plurithin there. We will also use the following form of the thinness theorem: if $E$ is plurithin at an accumulation point $a$, there is $u\in\calL_1$ with
\begin{equation}\label{eq:thin-barrier}
 u(a)>-\infty,\qquad
 \lim_{E\setminus\{a\}\ni z\to a}u(z)=-\infty.
\end{equation}
See \cite[Proposition~4.8.2 and Corollary~4.8.4]{Klimek1991}.

Negligibility and plurithinness at a point are different notions. For example, $N=\{0\}\times\cl\D\subset\C^2$ is negligible, since it is contained in $\{\log|z_1|=-\infty\}$, but it is not plurithin at $(0,0)$. Indeed, restriction to the complex line $z_1=0$ and the submean inequality exclude a strict jump at the center of the disk. Conversely, plurithinness at one point does not imply negligibility: the nonpolar planar compact $E$ in \cref{lem:planar} is thin at $2$. A set that is plurithin at \emph{each of its own points} is negligible \cite[Theorem~4.8.1]{Klimek1991}, but the converse fails in several complex variables. We always specify the point when using plurithinness.

The plurifine topology $\calF$ is the smallest topology in which all plurisubharmonic functions are continuous. Finite intersections of sets $B\cap\{u>c\}$, with $u\in\PSH(B)$, form a basis. A set $E$ is plurithin at $a$ exactly when $a$ is not a plurifine limit point of $E$; equivalently, some plurifine neighborhood of $a$ misses $E\setminus\{a\}$. In particular, the complement of a plurifine open neighborhood of $a$ is plurithin at $a$. See \cite[Theorems~4.8.7 and 4.8.9, Corollary~4.8.10]{Klimek1991}.

\section{Regular parameters}\label{sec:parameters}

We first prove concavity and two-sided estimates in the growth parameter. They determine the regular parameter sets and compare the right derivative with the unweighted Green function.

\begin{lemma}\label{lem:parameter-estimates}
For fixed $K$ and $\psi$, the following assertions hold.
\begin{enumerate}
\item For $\alpha,\beta>0$ and $0\le t\le1$,
\[
 F_{(1-t)\alpha+t\beta}\ge(1-t)F_\alpha+tF_\beta.
\]
\item If $0<\alpha<\beta$, then on $\C^n$,
\begin{equation}\label{eq:two-sided-slope}
 (\beta-\alpha)V_K^*
 \le F_\beta-F_\alpha
 \le \frac{\beta-\alpha}{\alpha}(F_\alpha-m).
\end{equation}
\end{enumerate}
\end{lemma}

\begin{proof}
If $u\in\calL_\alpha$ and $v\in\calL_\beta$ lie below $\psi$ on $K$, then $(1-t)u+tv\in\calL_{(1-t)\alpha+t\beta}$ and is again at most $\psi$ on $K$. Taking independent suprema gives concavity for the unregularized envelopes. Applying \eqref{eq:regularized-sum} proves the first assertion.

For the first inequality in \eqref{eq:two-sided-slope}, take $u\in\calL_\alpha$ with $u\le\psi$ on $K$ and $v\in\calL_1$ with $v\le0$ there. The function $u+(\beta-\alpha)v$ is admissible at parameter $\beta$. Thus
\[
 V_{\beta,K,\psi}\ge V_{\alpha,K,\psi}+(\beta-\alpha)V_K.
\]
Regularization and \eqref{eq:regularized-sum} prove the assertion.

For the upper bound, write $q=\psi-m\ge0$ on $K$. Scaling and monotonicity in the obstacle give
\[
 V_{\beta,K,q}
 =\beta V_{1,K,q/\beta}
 \le\beta V_{1,K,q/\alpha}
 =\frac\beta\alpha V_{\alpha,K,q}.
\]
After adding back $m$ and regularizing, this is the desired upper bound.

\end{proof}

Assume now that $\Lambda=\Lambda(K,\psi)$ is nonempty, and use its contact-ray description recalled in the introduction. For $a\in K$ and $\delta\in\Lambda$, define the regularity defect
\[
 d_\delta(a)=F_\delta(a)-\psi(a)\ge0.
\]

\begin{proof}[Proof of \cref{thm:main-rigidity}]
Fix $a\in K$. Contact gives $d_\delta(a)\ge0$ on $\Lambda$, and \cref{lem:parameter-estimates} shows that this function is concave and nondecreasing. Suppose that it vanishes at a parameter $\alpha$ that is not the smallest point of $\Lambda$. Choose $\eta\in\Lambda$ with $\eta<\alpha$. Since $\Lambda$ is a ray, every $\beta>\alpha$ also belongs to $\Lambda$. Write
\[
 \alpha=(1-t)\eta+t\beta,\qquad
 t=\frac{\alpha-\eta}{\beta-\eta}\in(0,1).
\]
Concavity gives
\[
 0=d_\alpha(a)\ge(1-t)d_\eta(a)+td_\beta(a)\ge0.
\]
Both terms in the sum are nonnegative and $t>0$, so $d_\beta(a)=0$. For every $\delta\in\Lambda$ with $\delta<\alpha$, monotonicity gives $0\le d_\delta(a)\le d_\alpha(a)=0$. Thus one interior zero forces vanishing on the entire contact ray. When $\Lambda=[\lambda,\infty)$, the remaining possibilities are no zero at all or a zero only at $\lambda$. When $\Lambda=(0,\infty)$, every parameter has a smaller parameter in $\Lambda$, so a single zero forces vanishing everywhere. Intersecting these pointwise possibilities over $a\in K$ gives the same alternatives for the global regular parameter set.
\end{proof}

An extension in $\calL_{\delta_1}^+$ gives contact for every $\delta\ge\delta_1$. The theorem therefore proves the endpoint assertion in Alan's question stated in the introduction. Also, if $V_K^*(a)>0$, the lower bound in \eqref{eq:two-sided-slope} makes $F_\delta(a)$ strictly increasing, so at most one parameter can be regular at $a$.

\begin{proposition}\label{prop:sensitivity}
Suppose $\tau\in\Lambda_{\reg}(a;K,\psi)$. The right derivative
\begin{equation}\label{eq:sensitivity}
 D_\tau(a)=\lim_{h\downarrow0}
 \frac{F_{\tau+h}(a)-\psi(a)}h
\end{equation}
exists and is finite. The quotient in this limit is nonincreasing in $h>0$, and
\[
 V_K^*(a)\le\frac{F_{\tau+h}(a)-\psi(a)}h
 \le D_\tau(a)\le\frac{\psi(a)-m}{\tau}.
\]
Its limit as $h\to\infty$ is $V_K^*(a)$, with the error bound
\[
 0\le\frac{F_{\tau+h}(a)-\psi(a)}h-V_K^*(a)
 \le\frac{M-m}{h}.
\]
Moreover,
\[
 \begin{aligned}
 D_\tau(a)=V_K^*(a)
 \quad\Longleftrightarrow\quad
 &F_\delta(a)=\psi(a)+(\delta-\tau)V_K^*(a)\\
 &\text{for every }\delta\ge\tau.
 \end{aligned}
\]
Regularity persists for every $\delta>\tau$ if and only if $D_\tau(a)=0$. If $D_\tau(a)>0$, regularity fails at every $\delta>\tau$, and $\tau$ is the smallest contact parameter.
\end{proposition}
\begin{proof}
Since $d_\tau(a)=0$, concavity gives, for $0<h_1<h_2$,
\[
 d_{\tau+h_1}(a)\ge\frac{h_1}{h_2}d_{\tau+h_2}(a).
\]
Thus the quotient is nonincreasing. The two-sided estimate \eqref{eq:two-sided-slope} bounds it between $V_K^*(a)$ and $(\psi(a)-m)/\tau$. This proves existence and finiteness of the derivative, as well as the first displayed bounds.

For the error bound, \cref{eq:two-sided-slope,eq:sandwich} give
\[
 0\le F_{\tau+h}(a)-\psi(a)-hV_K^*(a)
 \le M+\tau V_K^*(a)-\psi(a)\le M-m.
\]
The last inequality uses $\psi(a)=F_\tau(a)\ge\tau V_K^*(a)+m$. Division by $h$ also gives the limit at infinity. If $D_\tau(a)=V_K^*(a)$, the lower and upper bounds for every quotient agree, giving the affine formula. The converse follows by differentiation.

If $D_\tau(a)=0$, nonnegativity forces every defect to vanish. If $D_\tau(a)>0$, the defect is positive for small $h>0$ and hence, by monotonicity, for every $h>0$. The contact-threshold assertion follows from \cref{thm:main-rigidity}.
\end{proof}

Thus $V_K^*(a)$ is the limiting slope and $D_\tau(a)$ is the initial slope. Global pluriregularity makes the limiting slope zero, but need not make the initial slope zero; see \cref{thm:global-example}.

\subsection{A representation by measures}

The preceding criterion uses the weighted envelopes at all later parameters. We now give a criterion that uses the original weight and tests from $\calL_1$. It also gives a formula for the right derivative. In this subsection the weights are continuous, and we abbreviate quasi-everywhere as q.e.

Fix $a\in K$. Let $\mathcal J_a(K)$ be the set of pairs $(\mu,c)$, where $\mu$ is a Borel probability measure on $K$ and $c\ge0$, with the following property:
\begin{equation}\label{eq:J}
 u(a)\le\int_K f\,d\mu+c
\end{equation}
whenever $u\in\calL_1$, $f\in C(K)$, and $u\le f$ q.e. on $K$. Here $C(K)$ denotes the real-valued continuous functions on $K$.
The set $\mathcal J_a(K)$ depends only on $K$ and $a$, not on the weight or the growth parameter. We use pairs because the constant $c$ accounts for the logarithmic growth. The q.e. condition is essential at irregular points.

\begin{theorem}\label{thm:dual}
For every $q\in C(K)$ and $\delta>0$,
\[
 V_{\delta,K,q}^*(a)
 =\min_{(\mu,c)\in\mathcal J_a(K)}
 \left(\int_K q\,d\mu+\delta c\right).
\]
In particular, $\mathcal J_a(K)$ is nonempty, and
\begin{equation}\label{eq:min-c}
 \min_{(\mu,c)\in\mathcal J_a(K)}c=V_K^*(a).
\end{equation}
\end{theorem}

\begin{proof}
We first recall the q.e. form of the extremal envelope:
\begin{equation}\label{eq:qe-envelope}
 V_{\delta,K,f}^*(z)
 =\sup\{u(z):u\in\calL_\delta,\ u\le f\text{ q.e. on }K\}.
\end{equation}
This holds for every $z\in\C^n$, $f\in C(K)$, and $\delta>0$. By scaling, it is enough to prove it for $\delta=1$. The regularized extremal function itself belongs to the family on the right, since the unregularized envelope agrees with its regularization off a pluripolar set. Conversely, suppose $u\le f$ on $K\setminus P$, where $P$ is pluripolar. Choose $v\in\calL_1$ with $v=-\infty$ on $P$, and choose a constant $C$ such that $v-C\le f$ on $K$. For $0<\varepsilon<1$,
\[
 u_\varepsilon=(1-\varepsilon)u+\varepsilon(v-C)
\]
belongs to $\calL_1$ and is at most $f$ everywhere on $K$. Letting $\varepsilon\downarrow0$ off $\{v=-\infty\}$ gives $V_{1,K,f}^*\ge u$ there. The inequality holds everywhere by plurisubharmonic regularization. This proves \eqref{eq:qe-envelope}. The logarithmic-growth pluripolar barrier and the negligibility facts used here are standard; see \cite[Theorem~5.2.4 and Section~4.7]{Klimek1991} and \cite{Levenberg2017}.

Define the functional $T:C(K)\to\R$ by
\[
 T(f)=V_{1,K,f}^*(a).
\]
It is finite, monotone, concave, and satisfies
\[
 T(f+t)=T(f)+t,\qquad
 |T(f)-T(g)|\le\norm{f-g}_{C(K)}.
\]
Concavity follows directly from \eqref{eq:qe-envelope} by taking convex combinations of competitors. The other properties follow from translation and monotonicity of the obstacle.

Fix $f_0\in C(K)$. The supporting-functional form of the Hahn--Banach theorem gives a continuous linear functional $\ell$ such that
\begin{equation}\label{eq:support}
 T(f)\le T(f_0)+\ell(f-f_0),\qquad f\in C(K).
\end{equation}
Monotonicity shows that $\ell$ is positive, and the translation identity gives $\ell(1)=1$. By the Riesz representation theorem, $\ell(f)=\int_K f\,d\mu$ for a Borel probability measure $\mu$ on $K$. Set
\[
 c=T(f_0)-\int_K f_0\,d\mu.
\]
Taking $f=0$ in \eqref{eq:support} gives $c\ge T(0)=V_K^*(a)\ge0$. If $u\le f$ q.e. on $K$, then \eqref{eq:qe-envelope} and \eqref{eq:support} give
\[
 u(a)\le T(f)\le\int_K f\,d\mu+c.
\]
Thus $(\mu,c)\in\mathcal J_a(K)$, and $T(f_0)=\int_K f_0\,d\mu+c$.

For an arbitrary pair in $\mathcal J_a(K)$, apply \eqref{eq:J} to
$u=V_{1,K,f_0}^*$ and $f=f_0$. It gives the opposite inequality. Therefore
\[
 T(f_0)=\min_{(\mu,c)\in\mathcal J_a(K)}
 \left(\int_K f_0\,d\mu+c\right).
\]
Now use $V_{\delta,K,q}^*=\delta V_{1,K,q/\delta}^*$. The case $q=0$, $\delta=1$ proves \eqref{eq:min-c}.
\end{proof}

\begin{theorem}\label{thm:derivative}
Let $\psi\in C(K)$ and $\tau\in\Lambda_{\reg}(a;K,\psi)$. Then
\begin{equation}\label{eq:derivative}
 D_\tau(a)=\min\left\{c:(\mu,c)\in\mathcal J_a(K),\quad
 \int_K\psi\,d\mu+\tau c=\psi(a)\right\}.
\end{equation}
In particular, $D_\tau(a)=V_K^*(a)$ if and only if there is a probability measure $\mu$ on $K$ such that
\[
 (\mu,V_K^*(a))\in\mathcal J_a(K),\qquad
 \int_K\psi\,d\mu=\psi(a)-\tau V_K^*(a).
\]
This criterion uses only $K$, the weight $\psi$, and unweighted tests from $\calL_1$. In particular, it applies when $\psi=\Psi|_K$ with $\Psi\in\calL_\tau^+\cap C(\C^n)$ and $F_\tau(a)=\psi(a)$. No assumption of strict plurisubharmonicity is needed.
\end{theorem}

\begin{proof}
The set $\mathcal J_a(K)$ is closed for weak convergence of probability measures and ordinary convergence of $c$: each inequality in \eqref{eq:J} is closed because $f$ is continuous. Since the probability measures on the compact set $K$ form a weakly compact space, the pairs with $c\le C$ form a compact set for every $C<\infty$.

By \cref{thm:dual}, the pairs satisfying the equality in \eqref{eq:derivative} exist. They form a compact set: if $m=\min_K\psi$, their constants satisfy
\[
 0\le c\le\frac{\psi(a)-m}{\tau}.
\]
Thus the minimum on the right of \eqref{eq:derivative} is attained. Any such pair gives
\[
 F_{\tau+h}(a)\le\int_K\psi\,d\mu+(\tau+h)c
 =\psi(a)+hc,\qquad h>0.
\]
Consequently $D_\tau(a)$ is at most this minimum. Existence and finiteness were proved in \cref{prop:sensitivity}.

For each $h>0$, choose a minimizing pair $(\mu_h,c_h)$ for $F_{\tau+h}(a)$. The representation at $\tau$ gives
\[
 \psi(a)\le\int_K\psi\,d\mu_h+\tau c_h
 =F_{\tau+h}(a)-hc_h.
\]
Hence
\[
 0\le c_h\le\frac{F_{\tau+h}(a)-\psi(a)}h.
\]
The constants are bounded as $h\downarrow0$. Along a subsequence, $(\mu_h,c_h)$ converges to $(\mu,c)\in\mathcal J_a(K)$, and
\[
 \int_K\psi\,d\mu+\tau c=\psi(a),\qquad c\le D_\tau(a).
\]
This proves the reverse inequality in \eqref{eq:derivative}. Finally, every pair in $\mathcal J_a(K)$ has $c\ge V_K^*(a)$ by \eqref{eq:min-c}; the attained minimum gives the stated equality criterion.
\end{proof}

The q.e. qualification in \eqref{eq:J} cannot be replaced by an everywhere obstacle condition. With that replacement, $(\delta_a,0)$ would always be an admissible pair, where $\delta_a$ is the unit point mass. At $a\in K$, the resulting infimum would never exceed $\psi(a)$ and would miss the regularization defect. Formula \eqref{eq:J} retains the negligible-set information that is needed for $F_\delta$.

\section{Local regularity}\label{sec:local}

We recall the relation between local weighted and unweighted regularity, and then give an estimate under a local strict plurisubharmonicity assumption. The estimate explains why strictness near the point prevents the failure in \cref{thm:main-example,thm:global-example}.

For $\psi\in C(K)$ and $\delta\in\Lambda(K,\psi)$, local $(\delta,\psi)$-regularity is equivalent to local pluriregularity; this is \cite[Theorem~4.1]{Narzillaev2021}. The related continuity questions are studied in \cite{Dieu2003}. We will use the following consequence, for which upper semicontinuity is enough.

\begin{lemma}\label{prop:local-weights}
Let $\psi$ be bounded and upper semicontinuous on $K$, and let $\delta\in\Lambda(K,\psi)$. If $a$ is locally pluriregular, then $K$ is globally and locally $(\delta,\psi)$-regular at $a$.
\end{lemma}
\begin{proof}
Put $K_r=K\cap\cl B(a,r)$ and $M_r=\sup_{K_r}\psi$. Monotonicity in the set and \eqref{eq:sandwich} give
\[
 \psi(a)\le F_\delta(a)
 \le V_{\delta,K_r,\psi}^*(a)
 \le M_r+\delta V_{K_r}^*(a)=M_r.
\]
Upper semicontinuity gives $M_r\downarrow\psi(a)$. For local weighted regularity, apply the same argument to each $K_s$, using $r<s$ and the inherited contact \eqref{eq:inherit-contact}.
\end{proof}

The next theorem needs strictness only near the point. A global extension in the prescribed growth class is still required.

\begin{theorem}\label{thm:local-strict}
Let $a\in K$, let $\Psi\in\calL_\tau\cap C(\C^n)$, and put $\psi=\Psi|_K$. Assume that, for some $R>0$ and $c>0$,
\[
 \ddc\Psi\ge c\omega_0\quad\text{on }B(a,R).
\]
For $0<r<R$ and every $\delta\ge\tau$,
\[
 F_\delta(a)-\Psi(a)
 \ge cr^2\left(1+h_{K\cap\cl B(a,r),B(a,R)}^*(a)\right).
\]
Consequently, $(\delta,\psi)$-regularity at $a$ for one $\delta\ge\tau$ is equivalent to local pluriregularity at $a$. If it holds, it holds for every $\delta\ge\tau$.
\end{theorem}
\begin{proof}
Put $B=B(a,R)$, $E=K\cap\cl B(a,r)$, and take a competitor $u$ from \eqref{eq:relative}. On $B$, the function
\[
 H_u(z)=\Psi(z)-c\norm{z-a}^2+cr^2(u(z)+1)
\]
is plurisubharmonic, because
\[
 \ddc H_u=\ddc\Psi-c\omega_0+cr^2\ddc u\ge0.
\]
Since $u\le0$, we have $H_u\le\Psi$ whenever $\norm{z-a}\ge r$. Hence
\[
 W_u(z)=
 \begin{cases}
 \max\{\Psi(z),H_u(z)\},&z\in B,\\
 \Psi(z),&z\notin B
 \end{cases}
\]
is globally plurisubharmonic and belongs to $\calL_\tau$. On $E$, $u=-1$ and $H_u\le\Psi$; on $K\setminus E$ the same inequality follows from $\norm{z-a}\ge r$. Thus $W_u=\psi$ on $K$ and is a competitor for every $\delta\ge\tau$.

Taking the supremum over $u$ in a neighborhood of $a$ and then regularizing at $a$, using continuity of $\Psi$, gives the stated estimate. If the left side vanishes, all the relative extremal values equal $-1$, and the characterization recalled in \cref{sec:preliminaries} gives local pluriregularity. The converse and persistence follow from \cref{prop:local-weights}, since $[\tau,\infty)\subset\Lambda$.
\end{proof}

Under the local strictness hypothesis, initial regularity therefore gives
\[
 D_\tau(a)=V_K^*(a)=0.
\]
Indeed, persistence gives $D_\tau(a)=0$, and local pluriregularity gives $V_K^*(a)=0$. If regularity at $\tau$ fails to persist for a $C^2$ extension, then $(\ddc\Psi)^n(a)=0$: otherwise continuity of the positive form would imply strictness near $a$.

Sadullaev's global-to-local theorem assumes a globally strictly plurisubharmonic extension \cite[Theorem~2.4]{Sadullaev2016}; Alan assumes strictness on a ball containing $K$ \cite[Theorem~3.14]{Alan2019}. The proof above needs strictness only near $a$, but still requires a global extension in $\calL_\tau$. Initial regularity is essential: if $F_\tau(a)>\psi(a)$, the quotient in \eqref{eq:sensitivity} tends to $+\infty$ and does not represent a finite derivative at $\tau$.

\section{Plurithinness and negligible sets}\label{sec:fine}

Throughout this section, $K$ is compact and non-pluripolar at every point. We use the decomposition $K=R\cup T\cup S$ from \eqref{eq:RTS}. We first describe failure of local pluriregularity by removing a negligible subset. We then distinguish $T$ and $S$ in the relative plurifine topology.

\begin{lemma}\label{lem:irregular-negligible}
For $a\in K$, the following conditions are equivalent:
\begin{enumerate}
\item $a\notin R$;
\item there is a negligible set $N\subset K$ containing $a$ such that $K\setminus N$ is plurithin at $a$;
\item there is a plurifine open neighborhood $U$ of $a$ such that $K\cap U$ is negligible.
\end{enumerate}
\end{lemma}
\begin{proof}
Suppose first that $a\notin R$. Choose $r>0$ such that
\[
 g=V_{K\cap\cl B(a,r)}^*,\qquad c=g(a)>0.
\]
The local non-pluripolarity assumption ensures that $g$ is finite and plurisubharmonic. Set
\[
 U=B(a,r)\cap\{g>c/2\},\qquad N=K\cap U.
\]
Then $U$ is a plurifine open neighborhood of $a$. The unregularized extremal function is zero on $K\cap\cl B(a,r)$, so
\[
 N\subset\{V_{K\cap\cl B(a,r)}<V_{K\cap\cl B(a,r)}^*\}.
\]
Thus $N$ is negligible. Moreover, near $a$ the function $g$ is at most $c/2$ on $K\setminus N$, while $g(a)=c$. Hence $K\setminus N$ is plurithin at $a$. This proves both (ii) and (iii).

Suppose (ii) holds, and put $E=K\setminus N$. The point $a$ is an accumulation point of $E$. Otherwise some neighborhood of $a$ would meet $K$ only in the pluripolar set $N$, contrary to the local non-pluripolarity assumption. By \eqref{eq:thin-barrier}, there is $u\in\calL_1$ such that $u(a)$ is finite and $u(z)\to-\infty$ as $E\ni z\to a$. Choose $\rho>0$ so small that
\[
 u\le u(a)-1\quad\text{on }E\cap\cl B(a,\rho).
\]
The function $v=u-u(a)+1$ belongs to $\calL_1$, is nonpositive on $K\cap\cl B(a,\rho)\setminus N$, and satisfies $v(a)=1$. By \eqref{eq:polar-removal}, local pluriregularity is ruled out by the estimate
\[
 V_{K\cap\cl B(a,\rho)}^*(a)
 =V_{K\cap\cl B(a,\rho)\setminus N}^*(a)\ge1.
\]
Finally, (iii) implies (ii) with $N=K\cap U$, since $K\setminus N$ is contained in the complement of a plurifine neighborhood of $a$.
\end{proof}

\begin{proof}[Proof of \cref{thm:main-S}]
Let $a\in S$. The preceding lemma gives a negligible set $N\subset K$ containing $a$ such that $K\setminus N$ is plurithin at $a$. If $N$ were plurithin there as well, their union $K$ would be plurithin, a contradiction. Thus $N$ is non-plurithin at $a$.

Conversely, suppose such an $N$ exists. Since $N\subset K$ is non-plurithin at $a$, the set $K$ is non-plurithin there. The lemma shows that $a\notin R$. Hence $a\in S$.
\end{proof}

The criterion also describes regular points: $a\in R$ if and only if $K\setminus N$ is non-plurithin at $a$ for every negligible subset $N\subset K$. To see this, a set $N$ not containing $a$ can be replaced by $N\cup\{a\}$ without changing thinness at $a$.

We now define a negligible subset directly from $K$ and the plurifine topology:
\begin{equation}\label{eq:NK}
 N_K=\bigcup\{K\cap U:U\in\calF,\ K\cap U\text{ is negligible}\}.
\end{equation}
For any set $A$, let $A'_{\calF}$ denote its plurifine derived set:
\[
 A'_{\calF}
 =\{a:U\cap(A\setminus\{a\})\ne\varnothing
       \text{ for every plurifine neighborhood }U\text{ of }a\}.
\]
Let $\Iso_{\calF}(K)$ denote the isolated points of $K$ in its relative plurifine topology.

\begin{proposition}\label{prop:canonical-S}
The set $N_K$ is the largest relatively plurifine open negligible subset of $K$. Moreover,
\[
 R=K\setminus N_K,\qquad
 T=\Iso_{\calF}(K),\qquad
 S=N_K\cap(N_K)'_{\calF}.
\]
In particular, $S$ consists of the points of $N_K$ that are not isolated in the relative plurifine topology of $K$.
\end{proposition}
\begin{proof}
By \cref{lem:irregular-negligible}, $N_K=K\setminus R$. We verify that this set is negligible; an arbitrary union in \eqref{eq:NK} cannot be assumed negligible without proof.

Let $B_j$ range over the balls with rational centers and positive rational radii for which $K\cap B_j\ne\varnothing$. Put
\[
 A_j=K\cap\cl B_j,\qquad
 N_j=K\cap B_j\cap\{V_{A_j}^*>0\}.
\]
Each $A_j$ is non-pluripolar, and each $N_j$ is negligible. We claim that
\[
 K\setminus R=\bigcup_{j\ge1}N_j.
\]
If $a\notin R$, choose $r$ with $V_{K\cap\cl B(a,r)}^*(a)>0$ and then a rational ball $B_j$ with $a\in B_j$ and $\cl B_j\subset B(a,r)$. Monotonicity gives $V_{A_j}^*(a)>0$, so $a\in N_j$. Conversely, if $a\in N_j$, choose $\rho>0$ with $\cl B(a,\rho)\subset B_j$. Then
\[
 V_{K\cap\cl B(a,\rho)}^*(a)\ge V_{A_j}^*(a)>0,
\]
so $a\notin R$. This proves the countable-union representation and negligibility of $N_K$.

Relative plurifine openness and maximality now follow from \eqref{eq:NK}. The thinness criterion recalled in \cref{sec:preliminaries} gives $T=\Iso_{\calF}(K)$. Every such isolated point belongs to $N_K$, since its singleton is a relatively plurifine open pluripolar set. At a point of $N_K$, isolation in $N_K$ is equivalent to isolation in $K$, because $N_K$ is relatively plurifine open. Removing these isolated points from $N_K=K\setminus R$ gives the formula for $S$.
\end{proof}

In one complex variable, polar sets are thin at every point, so \cref{thm:main-S} gives $S=\varnothing$. In higher dimension a negligible set can be non-plurithin at its points; the product in \cref{sec:construction} provides an example.

\section{A product example}\label{sec:construction}

We construct the compact set and weight in \cref{thm:main-example}. The product geometry is taken from \cite[Example~3.16]{Alan2019}. We specify a planar compact set with a quantitative Green-function estimate, choose a smooth radial weight, and then study the parameter derivative on the exceptional polydisc.

\subsection{A planar compact set}

\begin{lemma}\label{lem:planar}
Let $a=2$, $p_j=2+2^{-j}$, $\rho_j=\exp(-16\cdot2^j)$, and
\begin{equation}\label{eq:E}
 E=\cl\D\cup\{2\}\cup\bigcup_{j\ge1}\cl D(p_j,\rho_j).
\end{equation}
Then $E$ is compact, full (its complement is connected), and nonpolar at every point. Its only point that is not locally regular is $2$, the set $E$ is thin at $2$, and
\begin{equation}\label{eq:planar-gap}
 V_E^*(2)\ge\frac49\log2.
\end{equation}
\end{lemma}
\begin{proof}
The disks in \eqref{eq:E} are pairwise disjoint, lie to the right of $2$, and accumulate only at $2$, since $\rho_j<2^{-j-3}$. Thus $E$ is closed and bounded. Every neighborhood of $2$ contains one of these disks, and all other points belong to a disk with nonempty interior. Hence $E$ is nonpolar at every point. At points other than $2$, local regularity follows from the local regularity of a closed disk and monotonicity of the extremal function.

All disk centers lie on the real axis. A point of $\C\setminus E$ in the upper half-plane can be moved vertically upward without hitting any disk; points in the lower half-plane can similarly be moved downward. Points on the real axis outside $E$ can be moved upward. Above and below all disks, horizontal paths can be joined along the line $\operatorname{Re}z=-2$, which misses $E$. Thus $\C\setminus E$ is path connected, so $E$ is full.

Consider the logarithmic potential
\[
 h(z)=\sum_{j\ge1}2^{-j}\log|z-p_j|+2\log2.
\]
The positive parts are uniformly summable on compact sets. Subtracting a common upper bound on each compact gives a decreasing series of subharmonic functions, and its limit is not identically $-\infty$, since
\[
 h(2)=\sum_{j\ge1}2^{-j}(-j\log2)+2\log2=0.
\]
Consequently $h$ is globally subharmonic. Its growth at infinity is $\log|z|+O(1)$ because the atomic measure has total mass one and compact support.

If $z\in\cl D(p_j,\rho_j)$, then $|z-p_k|<1$ for $k\ne j$. Therefore
\begin{equation}\label{eq:h-small}
 h(z)\le2^{-j}\log\rho_j+2\log2=-16+2\log2.
\end{equation}
Since only these disks accumulate at $2$, \eqref{eq:h-small} and $h(2)=0$ prove thinness at $2$.

It remains to prove \eqref{eq:planar-gap}, which also excludes local regularity at $2$. On $\cl\D$, $|z-p_j|<4$, so $h(z)\le4\log2$. Define
\[
 v(z)=\frac{8\lp|z|+h(z)-4\log2}{9}\in\calL_1.
\]
Then $v\le0$ on $\cl\D$. On every small disk, $|z|<4$ and \eqref{eq:h-small} give
\[
 9v(z)\le14\log2-16<0.
\]
Thus $v\le0$ on $E\setminus\{2\}$, while $v(2)=\frac49\log2$. Since a singleton is polar, \eqref{eq:polar-removal} gives
\[
 V_E^*(2)=V_{E\setminus\{2\}}^*(2)\ge v(2)=\frac49\log2.\qedhere
\]
\end{proof}

\subsection{A product set}

For $n\ge2$ set
\begin{equation}\label{eq:Kn}
 K_n=E\times\cl\D^{\,n-1},\qquad
 Z_n=\{2\}\times\cl\D^{\,n-1}.
\end{equation}

\begin{lemma}\label{lem:product}
The compact $K_n$ is polynomially convex, non-pluripolar at every point, and non-plurithin at every point. In the decomposition \eqref{eq:RTS},
\[
 T=\varnothing,\qquad R=K_n\setminus Z_n,\qquad S=Z_n.
\]
Moreover $V_{K_n}^*(2,w)=V_E^*(2)$ for all $w\in\cl\D^{\,n-1}$.
\end{lemma}
\begin{proof}
A full planar compact is polynomially convex; the elementary planar regularity facts used here may be found in \cite{Ransford1995}. Products of polynomially convex compacts are polynomially convex, by separation in a coordinate outside the relevant factor. This proves the first assertion.

Every neighborhood of every point of $K_n$ contains a product of nonpolar planar pieces, and indeed contains an open subset of $\C^n$: at a point of $Z_n$ use one of the sufficiently small disks approaching $2$, together with interior pieces of the other factors. Thus $K_n$ is non-pluripolar at each point.

Fix $b=(z,w)\in K_n$. Vary one coordinate of $w$ in its closed unit disk, keeping all other coordinates fixed. If a local psh function has finite value at $b$, its restriction to this complex line is subharmonic and not identically $-\infty$. A closed disk is non-thin at every one of its points, including its boundary. Hence there can be no strict jump along $K_n$ at $b$. If the psh function has value $-\infty$ at $b$, a strict jump is impossible as well. This proves non-plurithinness everywhere.

At a point outside $Z_n$, products of small regular disk pieces give local pluriregularity; one may use the standard product formula
\begin{equation}\label{eq:product}
 V_{K_n}^*(z,w)=\max\{V_E^*(z),\lp|w_1|,\ldots,\lp|w_{n-1}|\};
\end{equation}
see \cite{Klimek1991,Levenberg2017}. More explicitly, each sufficiently small neighborhood contains a product of intersections of closed disks centered at, or containing, the corresponding coordinates; these planar pieces are regular at the point, and their product is regular there. At a point $(2,w)$, \eqref{eq:product} and \cref{lem:planar} give $V_{K_n}^*(2,w)=V_E^*(2)>0$. Such a point is not even globally pluriregular. The stated decomposition follows. It also illustrates \cref{thm:main-S} directly: $Z_n$ is negligible because it lies in a complex hyperplane, and is non-plurithin at its points by the disk argument above. The planar thinness barrier at $2$, lifted through the first coordinate, shows that $K_n\setminus Z_n$ is plurithin at every point of $Z_n$. Thus $N_{K_n}=S=Z_n$.
\end{proof}

\subsection{A smooth weight}

\begin{lemma}\label{lem:smooth}
There exist $b>0$ and a radial function $\Phi\in C^\infty(\C^n)\cap\calL_1^+$ such that
\begin{enumerate}
\item $\Phi$ is plurisubharmonic and $0\le\Phi\le b$ on $\cl B(0,1)$;
\item $\Phi(z)=\log\norm z+b$ for $\norm z\ge1$, and $\Phi(z)\ge b+\log\norm z$ for $z\ne0$;
\item $\rank(\ddc\Phi)=n-1$ on $\{\norm z>1\}$.
\end{enumerate}
\end{lemma}
\begin{proof}
Choose a smooth nondecreasing function $s:\R\to[0,1]$ equal to $0$ for $t\le-1$ and to $1$ for $t\ge0$, with all derivatives matching the constant pieces. Put
\[
 \chi(t)=\int_{-1}^t s(x)\,dx,\qquad b=\int_{-1}^0s(x)\,dx,
\]
where the integral is zero when $t\le-1$. Then $\chi$ is smooth, convex, and nondecreasing, equals $0$ on $(-\infty,-1]$, and equals $t+b$ on $[0,\infty)$. Define $\Phi(z)=\chi(\log\norm z)$ for $z\ne0$ and $\Phi(0)=0$. Convex increasing composition preserves plurisubharmonicity. The function is constant near $0$, so it is smooth there as well as everywhere else. Its growth and bounds follow from the definition. In particular, $0\le\chi'\le1$ and $\chi(t)=t+b$ for $t\ge0$ imply $\chi(t)\ge t+b$ for all $t$. Finally, for $z\ne0$,
\[
 \ddc\log\norm z
 =\frac{i}{4}\sum_{j,k=1}^n
 \left(\frac{\delta_{jk}}{\norm z^2}
       -\frac{\bar z_jz_k}{\norm z^4}\right)
 dz_j\wedge d\bar z_k,
\]
where $\delta_{jk}$ is the Kronecker symbol. By the Cauchy--Schwarz inequality, this form is nonnegative, and its kernel is the complex radial direction $\C z$. Since $\ddc\Phi=\ddc\log\norm z$ for $\norm z>1$, assertion~\textup{(iii)} follows. In particular, for $n\ge2$ its exterior powers satisfy $(\ddc\Phi)^n=0$ and $(\ddc\Phi)^{n-1}\ne0$ there.
\end{proof}

The same radial function fixes the initial envelope for any compact set $K$ containing the closed unit ball $A=\cl B(0,1)$. Put $\Psi_\tau=\tau\Phi$ and $\psi_\tau=\Psi_\tau|_K$. Then
\begin{equation}\label{eq:base-identity}
 V_{\tau,K,\psi_\tau}=F_\tau=\Psi_\tau
 \quad\text{on }\C^n.
\end{equation}
Indeed, $\Psi_\tau$ is admissible. Any other admissible $u$ is at most $\tau b$ on $A$, so $V_A=\lp\norm z$ gives
\[
 u(z)\le\tau b+\tau\lp\norm z.
\]
Outside $A$ this is $\Psi_\tau(z)$; on $A$ the obstacle gives $u\le\Psi_\tau$ directly. For every $\delta\ge\tau$, the same extension gives contact on $K$. If $K$ contains a point $a$ with $\norm a>1$, then for $0<\delta<\tau$,
\[
 V_{\delta,K,\psi_\tau}(a)
 \le\tau b+\delta\log\norm a
 <\tau b+\tau\log\norm a=\psi_\tau(a).
\]
Consequently,
\begin{equation}\label{eq:radial-contact}
 \Lambda(K,\psi_\tau)=[\tau,\infty)
 \quad\text{whenever }K\supset A\text{ and }K\setminus A\ne\varnothing.
\end{equation}
We use these two identities in both examples.

\begin{proof}[Proof of \cref{thm:main-example}]
Take $K_n$ from \eqref{eq:Kn}, and put $\Psi_\tau=\tau\Phi$ and $\psi_\tau=\Psi_\tau|_{K_n}$. Since $E\supset\cl\D$, the compact $K_n$ contains the closed unit ball. It also contains $(2,0,\ldots,0)$. Thus \cref{eq:base-identity,eq:radial-contact} give the initial envelope and the contact ray.

For $\delta>\tau$, apply \eqref{eq:two-sided-slope}, \eqref{eq:base-identity}, and \cref{lem:product} to obtain, at every $a\in Z_n$,
\[
 V_{\delta,K_n,\psi_\tau}^*(a)-\psi_\tau(a)
 \ge(\delta-\tau)V_{K_n}^*(a)
 =(\delta-\tau)V_E^*(2)
 \ge\frac49(\delta-\tau)\log2.
\]
This proves \eqref{eq:main-gap}. At every point of $K_n\setminus Z_n$, \cref{prop:local-weights} gives regularity for all $\delta\ge\tau$. Hence
\[
 \Lambda_{\reg}(a;K_n,\psi_\tau)=
 \begin{cases}
 [\tau,\infty),&a\in K_n\setminus Z_n,\\
 \{\tau\},&a\in Z_n.
 \end{cases}
\]
The global regular parameter set is their intersection, namely $\{\tau\}$. Since $\norm a\ge2$ on $Z_n$, the assertion about $\rank(\ddc\Psi_\tau)$ follows from \cref{lem:smooth}.
\end{proof}

\subsection{Estimates and exact formulas}

Restrictions to complex lines give an upper bound for the defect throughout $Z_n$. We first obtain an affine formula on a smaller polydisc. We then compute the right derivative at every point of $Z_n$.

For $w\in\cl\D^{\,n-1}$, put
\[
 q(w)=\max_{1\le j\le n-1}|w_j|,\qquad
 R_w=\begin{cases}2/q(w),&q(w)>0,\\+\infty,&q(w)=0,\end{cases}
 \qquad E_w=E\cap\cl D(0,R_w),
\]
with $E_w=E$ when $R_w=+\infty$. Each $E_w$ contains $\cl\D\cup\{2\}$ and is nonpolar.

Since $\max_{\zeta\in E}|\zeta|=5/2+e^{-32}$, attained on the first small disk, put
\[
 \rho_*=\frac{2}{\max_{\zeta\in E}|\zeta|}
 =\frac4{5+2e^{-32}}<\frac45.
\]

\begin{proposition}\label{prop:fibre-bounds}
For every $w\in\cl\D^{\,n-1}$ and $\delta>\tau$,
\begin{equation}\label{eq:fibre-bounds}
 (\delta-\tau)V_E^*(2)
 \le F_\delta(2,w)-\psi_\tau(2,w)
 \le(\delta-\tau)V_{E_w}^*(2).
\end{equation}
If $q(w)\le\rho_*$, then, for every $\delta\ge\tau$,
\[
 F_\delta(2,w)=\psi_\tau(2,w)+(\delta-\tau)V_E^*(2).
\]
On the central complex line, for every $\zeta\in\C$ and $\delta\ge\tau$,
\[
 F_\delta(\zeta,0)=\tau\Phi(\zeta,0)+(\delta-\tau)V_E^*(\zeta).
\]
\end{proposition}
\begin{proof}
Fix $w$, write $h=\delta-\tau>0$, and consider the complex line
\[
 \ell_w(\zeta)=\left(\zeta,\frac{\zeta}{2}w\right).
\]
Then $\ell_w(2)=(2,w)$ and $\ell_w(\zeta)\in K_n$ exactly when $\zeta\in E_w$. Define
\[
 u_h(\zeta)=F_{\tau+h}(\ell_w(\zeta))-\tau\Phi(\ell_w(\zeta)).
\]
By \eqref{eq:base-identity} and monotonicity, $u_h\ge0$. If $\zeta\in E_w\setminus\{2\}$, the point $\ell_w(\zeta)$ is locally pluriregular in $K_n$, so \cref{prop:local-weights} gives $u_h(\zeta)=0$. In particular, $u_h=0$ on $\cl\D$.

Put $c_w=(1+\norm w^2/4)^{1/2}$. For $|\zeta|>1$, we have $\norm{\ell_w(\zeta)}=c_w|\zeta|>1$, and hence
\[
 \Phi(\ell_w(\zeta))=\log|\zeta|+\log c_w+b.
\]
This function is harmonic on $\C\setminus\cl\D$. Therefore $u_h$ is subharmonic there. Its upper semicontinuity and its zero values on $\partial\D$ give a boundary upper limit at most zero. Since $u_h\ge0$, the subharmonic gluing lemma joins it to zero on $\cl\D$. Thus $u_h$ is subharmonic on $\C$, and its growth satisfies $u_h\in\calL_h(\C)$.

The function $u_h/h$ is nonpositive on $E_w\setminus\{2\}$. By the definition of the extremal function and \eqref{eq:polar-removal},
\begin{equation}\label{eq:line-upper}
 u_h\le hV_{E_w\setminus\{2\}}^*=hV_{E_w}^*\quad\text{on }\C.
\end{equation}
Evaluating at $2$ proves the upper bound in \eqref{eq:fibre-bounds}. The lower bound follows from \eqref{eq:two-sided-slope}, \eqref{eq:base-identity}, and $V_{K_n}^*(2,w)=V_E^*(2)$.

If $q(w)\le\rho_*$, then $E_w=E$, so the bounds in \eqref{eq:fibre-bounds} agree. The case $\delta=\tau$ follows from \eqref{eq:base-identity}.

For $w=0$, \eqref{eq:line-upper} holds with $E_w=E$ at every $\zeta$. The reverse inequality follows from \eqref{eq:two-sided-slope} and $V_{K_n}^*(\zeta,0)=V_E^*(\zeta)$, proving the central-line identity.
\end{proof}

On the whole exceptional polydisc, \cref{lem:planar,prop:fibre-bounds} give the dimension-independent bounds
\[
 \frac49(\delta-\tau)\log2
 \le F_\delta(2,w)-\psi_\tau(2,w)
 \le(\delta-\tau)\log2.
\]
The next result determines the initial slope throughout the exceptional polydisc.

\begin{theorem}\label{thm:product-derivative}
For every $w\in\cl\D^{\,n-1}$ and $\tau>0$,
\[
 D_\tau(2,w)=V_{E_w}^*(2).
\]
Consequently,
\[
 D_\tau(2,w)=V_E^*(2)
 \quad\Longleftrightarrow\quad q(w)\le\rho_*.
\]
For the remaining points,
\[
 \rho_*<q(w)\le1
 \quad\Longrightarrow\quad
 V_E^*(2)<D_\tau(2,w)\le\log2.
\]
At the boundary of the polydisc,
\[
 q(w)=1\quad\Longrightarrow\quad D_\tau(2,w)=\log2.
\]
The derivative is independent of $\tau$ for this family of weights.
\end{theorem}

We need a pointwise convergence fact. Ordinary convergence of Green functions outside a limiting compact does not by itself justify evaluation at the irregular point $2$.

\begin{lemma}\label{lem:green-limit}
For $R\ge2$, set $E_R=E\cap\cl D(0,R)$. Then
\[
 \lim_{s\downarrow R}V_{E_s}^*(2)=V_{E_R}^*(2).
\]
If $R<\max_{\zeta\in E}|\zeta|$, then
\[
 V_{E_R}^*(2)>V_E^*(2).
\]
\end{lemma}

\begin{proof}
All $E_R$ are full. Their components, apart from the accumulation point $2$, are a disk, closed lenses, or isolated tangent points. A nondegenerate lens is regular, and an isolated tangent point is polar and does not change the regularized Green function. The complements are connected: all the disks and lenses have centers or axes on the real line, so a point outside them in the upper or lower half-plane can be moved vertically away and then around the compact.

Suppose first that $R>2$. Choose $0<r<R-2$ so that the circle $C_r=\{\zeta:|\zeta-2|=r\}$ misses $E$. Such circles exist between consecutive small disks. On $D(2,r)$ all the sets $E_s$, $s\ge R$, agree with $E$. The standard convergence theorem for decreasing planar compacts gives
\[
 V_{E_s}^*\uparrow V_{E_R}^*
 \quad\text{locally uniformly on }\C\setminus E_R
 \quad(s\downarrow R).
\]
This follows, for example, from continuity from above of logarithmic capacity, weak convergence of equilibrium measures, and their logarithmic-potential representation away from the limiting compact; see \cite{Levenberg2017,Ransford1995}.
In particular,
\[
 \varepsilon_s:=\sup_{C_r}(V_{E_R}^*-V_{E_s}^*)\longrightarrow0.
\]
The difference is nonnegative and harmonic on $D(2,r)\setminus E$. Its boundary limit is zero at each point of $E\cap D(2,r)$ other than $2$. It is bounded near $2$. The maximum principle with one bounded exceptional boundary point gives
\[
 V_{E_R}^*\le V_{E_s}^*+\varepsilon_s
 \quad\text{on }D(2,r)\setminus E.
\]
One can justify this version of the maximum principle by adding
$\varepsilon\log(|\zeta-2|/r)$ to the difference and then letting $\varepsilon\downarrow0$. On $E\cap D(2,r)\setminus\{2\}$ both Green functions vanish. The inequality therefore extends to $2$ by subharmonic regularization. This proves convergence for $R>2$.

For $R=2$, we have $E_2=\cl\D\cup\{2\}$, so $V_{E_2}^*(2)=\log2$. To control the disks accumulating at $2$, put
\[
 T_J=\cl\D\cup\{2\}\cup\bigcup_{j\ge J}\cl D(p_j,\rho_j),
 \qquad m_J=\sum_{j\ge J}2^{-j}=2^{1-J},
\]
and define
\[
 v_J(\zeta)=
 \frac{\lp|\zeta|+\sum_{j\ge J}2^{-j}\log|\zeta-p_j|-m_J\log4}
 {1+m_J}.
\]
The function $v_J$ belongs to $\calL_1$. On $\cl\D$ the numerator is nonpositive. On a small disk with index $j\ge J$, the $j$th logarithmic term is at most $-16$, the other logarithms are nonpositive, and $\lp|\zeta|<\log4$. Hence $v_J\le0$ on $T_J\setminus\{2\}$. Polar removal gives
\[
 V_{T_J}^*(2)\ge v_J(2)
 =\frac{1-(J+3)m_J}{1+m_J}\log2
 \longrightarrow\log2.
\]
For every $J$, if $s>2$ is sufficiently close to $2$, then $E_s\subset T_J$. Since $V_{E_s}^*(2)\le\log2$, this proves convergence at $R=2$.

Finally, suppose $R<\max_E|\zeta|$. An open part of the first small disk is removed from $E_R$. At a boundary point of that disk with modulus greater than $R$, the function $V_{E_R}^*$ is positive, whereas $V_E^*$ is zero. Thus
\[
 d=V_{E_R}^*-V_E^*
\]
is positive somewhere in $\C\setminus E$. It is nonnegative and harmonic on that connected domain, so it is positive everywhere there. Choose $L>\max_E|\zeta|$. On $|\zeta|=L$, choose $\gamma>0$ such that $d\ge\gamma V_E^*$. The minimum principle on $D(0,L)\setminus E$ gives the same inequality throughout that domain: at each point of $E\setminus\{2\}$, the lower boundary limit of $d-\gamma V_E^*$ is nonnegative, and $2$ is a bounded polar exception. Extending the inequality by subharmonic regularization yields
\[
 V_{E_R}^*(2)\ge(1+\gamma)V_E^*(2)>V_E^*(2),
\]
where the last step uses \eqref{eq:planar-gap}.
\end{proof}

\begin{proof}[Proof of \cref{thm:product-derivative}]
Fix $w$ and set $a=(2,w)$. Dividing the upper bound in \cref{prop:fibre-bounds} by $h=\delta-\tau$ and letting $h\downarrow0$ gives
\[
 D_\tau(a)\le V_{E_w}^*(2).
\]
For $q(w)=0$, equality follows from \cref{prop:fibre-bounds}. Suppose now that $q(w)>0$, and write $R=R_w$. Choose $s>R$ different from every inner tangency radius $p_j-\rho_j$, and put $v_s=V_{E_s}^*$. Then $v_s=0$ on $E_s\setminus\{2\}$: the included pieces of the small disks are nondegenerate lenses or whole disks and are regular.

For $Z=(z_1,\ldots,z_n)$, use the complex linear function
\[
 L(Z)=\frac{2z_1+\sum_{j=1}^{n-1}\overline{w_j}z_{j+1}}
 {\sqrt{4+\norm w^2}}.
\]
We have $|L(Z)|\le\norm Z$, with equality exactly on the complex line through $a$, and $L(a)=\norm a$. The radial construction gives $\Phi(Z)\ge b+\log\norm Z$. Hence
\[
 \tau b+\tau\log|L(Z)|\le\tau\Phi(Z),
\]
with equality at $a$.

On the compact set $\{Z\in K_n:|z_1|\ge s\}$ this inequality has a uniform positive gap. Indeed, a point on the complex line through $a$ has the form $(\zeta,\zeta w/2)$. If it belongs to $K_n$, the polydisc constraint gives $|\zeta|\le R<s$. Thus the compact set misses this line, and $|L(Z)|/\norm Z$ has maximum strictly less than $1$ there. If the compact set is empty, no gap estimate is needed.

Since $0\le v_s(z_1)\le\lp|z_1|$ is bounded for $z_1\in E$, there is $h_s>0$ such that, for $0<h\le h_s$, the function
\[
 u_h(Z)=\tau b+\tau\log|L(Z)|+h v_s(z_1)
\]
is at most $\psi_\tau$ on $K_n\setminus Z_n$. On the part with $|z_1|\le s$ this follows from $v_s=0$; on the rest it follows from the positive gap. Also $u_h\in\calL_{\tau+h}$ and $Z_n$ is pluripolar. The quasi-everywhere envelope formula \eqref{eq:qe-envelope} therefore gives $F_{\tau+h}\ge u_h$ on $\C^n$. Evaluating at $a$ yields
\[
 F_{\tau+h}(a)-\psi_\tau(a)\ge hV_{E_s}^*(2).
\]
It follows that $D_\tau(a)\ge V_{E_s}^*(2)$. Let $s\downarrow R$ through radii avoiding the countably many tangencies. By \cref{lem:green-limit}, the right side tends to $V_{E_R}^*(2)=V_{E_w}^*(2)$. This proves the derivative formula.

If $q(w)\le\rho_*$, then $E_w=E$. If $q(w)>\rho_*$, the strict inequality in \cref{lem:green-limit} gives $V_{E_w}^*(2)>V_E^*(2)$. Finally, $q(w)=1$ gives $E_w=\cl\D\cup\{2\}$; polar removal gives $V_{E_w}^*(2)=\log2$. All assertions follow.
\end{proof}

All alternatives in \cref{thm:main-rigidity} occur. A regular closed ball with zero obstacle has $\Lambda_{\reg}=\Lambda=(0,\infty)$. The compact $K_n$ with zero obstacle has $\Lambda=(0,\infty)$ and empty global regular parameter set, because $V_{K_n}^*>0$ on $Z_n$. The smooth weighted construction realizes the singleton alternative at an arbitrarily prescribed positive threshold. Its points outside $Z_n$ simultaneously realize the full-ray alternative $[\tau,\infty)$.

\section{Globally pluriregular compact sets}\label{sec:global-example}

We now prove \cref{thm:global-example}. Unlike the product example, the compact set here satisfies $V_K^*|_K=0$. The construction uses a variant of Sadullaev's chair \cite[pp.~150--151]{Sadullaev2016}. A ball fixes the initial weighted envelope, and small product sets make the compact non-pluripolar at every point. We begin in $\C^2$ and then pass to higher dimensions.

\subsection{A planar compact set}

For $j\ge1$, set
\[
 q_j=3-2^{-j},\qquad r_j=e^{-16\cdot2^j},\qquad
 I_j=[q_j-r_j,q_j+r_j].
\]
These pairwise disjoint real intervals lie in $(2,3)$ and accumulate only at $3$. Define
\[
 I=\{3\}\cup\bigcup_{j\ge1}I_j,\qquad
 H=\overline\D\cup I.
\]
Both $I$ and $H$ are compact. Although $V_H^*$ vanishes on the unit disk and on every $I_j$, it is positive at $3$.

\begin{lemma}\label{lem:global-planar}
The function $V_H^*$ satisfies
\[
 0\le V_H^*(w)\le\lp|w|,\qquad
 V_H^*=0\text{ on }\overline\D\cup\bigcup_{j\ge1}I_j,\qquad
 V_H^*(3)\ge\frac{8\log3-4\log2}{9}>0.
\]
For $0<r\le1$, put $H_r=\overline D(0,r)\cup I$. Then
\begin{equation}\label{eq:global-shrunk}
 V_H\le V_{H_r}\le V_H+\log(1/r),\qquad
 V_H^*\le V_{H_r}^*\le V_H^*+\log(1/r).
\end{equation}
\end{lemma}
\begin{proof}
The first upper bound follows from $\overline\D\subset H$. A closed real interval is regular for the planar Green function. Monotonicity with respect to the compact set therefore gives $V_H^*=0$ on each $I_j$; the same argument applies to $\overline\D$.

Consider the logarithmic potential
\[
 P(w)=\sum_{j=1}^{\infty}2^{-j}\log|w-q_j|+2\log2.
\]
It is subharmonic, has logarithmic growth of coefficient $1$, and is not identically $-\infty$, since
\[
 P(3)=-\log2\sum_{j=1}^{\infty}j2^{-j}+2\log2=0.
\]
These facts also follow by viewing the sum as the logarithmic potential of the probability measure $\sum_j2^{-j}\delta_{q_j}$. On $\overline\D$, all the distances $|w-q_j|$ are at most $4$, so $P\le4\log2$. On $I_j$, the $j$th summand is at most $-16$, and all other logarithms are nonpositive. Consequently,
\[
 P\le-16+2\log2\quad\text{on }\bigcup_j I_j.
\]
It follows that
\[
 Q(w)=\frac{8\lp|w|+P(w)-4\log2}{9}\in\calL_1(\C)
\]
satisfies $Q\le0$ on $H\setminus\{3\}$ and $Q(3)=(8\log3-4\log2)/9$. Polar removal in \eqref{eq:polar-removal} gives $V_H^*\ge Q$, proving the lower bound at $3$.

Finally, $V_{H_r}\le\lp(|w|/r)$. Therefore $V_{H_r}\le\log(1/r)$ on $\overline\D$, while $V_{H_r}=0$ on $I$. Subtracting $\log(1/r)$ from each competitor for $V_{H_r}$ produces a competitor for $V_H$. This proves the upper raw inequality in \eqref{eq:global-shrunk}. The lower one follows from $H_r\subset H$, and regularization proves the remaining inequalities.
\end{proof}

\subsection{The compact set}

Use coordinates $(z,w)\in\C^2$, and put
\begin{align}
 A&=\{(z,w):|z|^2+|w|^2\le1\},\notag
\\
 \Gamma&=\{(e^{i\theta},t):-\pi\le\theta\le\pi,
                   \ 3\le t\le4-\cos\theta\},\notag
\\
 B&=(\overline\D\times I)\cup\Gamma,\qquad
 K=A\cup B,\qquad a=(0,3).\label{eq:global-K}
\end{align}
The set $\overline\D\times\{3\}$ is the base disk of the chair. The product pieces $\overline\D\times I_j$ approach that disk from below, while $\Gamma$ lies above it.

\begin{lemma}\label{lem:global-geometry}
The compact set $K$ is polynomially convex, non-pluripolar and non-plurithin at every point, and globally pluriregular.
\end{lemma}
\begin{proof}
We first prove polynomial convexity, then the local size properties, and finally global pluriregularity.

\emph{Polynomial convexity.}
The projection of $B$ onto the $w$-coordinate is
$W=I\cup[3,5]\subset\R$. For $t\in W$, its fiber is
\[
 B_t=\begin{cases}
 \overline\D,&t\in I,\\
 \{z:|z|=1,\ \operatorname{Re}z\le4-t\},&3<t\le5.
 \end{cases}
\]
Every fiber is polynomially convex in $\C$: it is a disk, a proper closed circular arc, or a singleton.

Here is a direct proof that $B$ is polynomially convex. A point in its polynomial hull must project into $W$, because a compact subset of the real axis is polynomially convex. If $(z_0,t_0)\notin B$ with $t_0\in W$, choose a polynomial $p$ in $z$ with $p(z_0)=1$ and $\sup_{B_{t_0}}|p|<1$. By compactness, $|p|\le\rho<1$ on fibers with $|t-t_0|$ sufficiently small. Since $W\subset(2,5]$, the polynomial
\[
 q(t)=1-\frac{(t-t_0)^2}{16}
\]
lies in $(0,1]$ on $W$, equals $1$ at $t_0$, and is uniformly less than $1$ away from that neighborhood. For large $k$, $p(z)q(w)^k$ has modulus less than $1$ on $B$ and equals $1$ at $(z_0,t_0)$. This separates the point from $B$.

The ball $A$ is polynomially convex. Its $w$-projection is $\overline\D$, whereas $W\subset(2,5]$. The planar compact $\overline\D\cup W$ has connected complement. By polynomial approximation in one variable, there are polynomials in $w$ tending uniformly to $1$ on either one of these two compact sets and to $0$ on the other. Multiplying a separating polynomial for $A$ or $B$ by such a polynomial separates any point over its corresponding projection. Points outside both projections are separated by a polynomial in $w$. Thus $A\cup B$ is polynomially convex.

\emph{Non-pluripolarity at every point.}
This holds on $A$ and on each $\overline\D\times I_j$, including their boundaries. At a point of $\overline\D\times\{3\}$, every neighborhood contains a product of a nonpolar disk piece and a nondegenerate subinterval of some $I_j$. At a point of $\Gamma$ with $t>3$, every neighborhood contains a product of a nondegenerate circular arc and a real interval. Products of nonpolar planar sets are non-pluripolar. This proves the assertion at every point of $K$.

Non-plurithinness is also direct. A point of $\cl\D\times I$ lies on a complex disk contained in $K$. A point of $\Gamma$ with $t>3$ lies on a nondegenerate real interval in a complex line. Disks and real intervals are non-thin at their endpoints as well as their interior points. Restricting a local plurisubharmonic function to the corresponding line rules out a strict jump. Points of $A$ are non-plurithin because the ball is locally pluriregular.

\emph{Global pluriregularity.}
The function $V_K^*$ is finite, nonnegative, and plurisubharmonic. Since a ball and products of regular planar compact sets are globally pluriregular, monotonicity gives
\[
 V_K^*=0\quad\text{on }A\cup\bigcup_j(\overline\D\times I_j).
\]
For $\theta\ne0$ and $3<t<4-\cos\theta$, there is a product of a small closed circular arc and a real interval, contained in $\Gamma$, that contains $(e^{i\theta},t)$ in the relative interior of both factors. The product formula for extremal functions gives $V_K^*(e^{i\theta},t)=0$.

For fixed $\theta\ne0$, the function $w\mapsto V_K^*(e^{i\theta},w)$ is finite and subharmonic. The non-thinness of a real interval at its endpoints extends this equality to $3\le t\le4-\cos\theta$. In particular,
\[
 V_K^*(e^{i\theta},3)=0\quad(\theta\ne0).
\]
The finite subharmonic function $z\mapsto V_K^*(z,3)$ is bounded near $\overline\D$ and vanishes on its boundary except possibly at $1$. The subharmonic Poisson inequality, or the maximum principle with a bounded exceptional boundary point, gives $V_K^*(z,3)=0$ for $|z|<1$. Non-thinness of the disk at $1$ then gives $V_K^*(1,3)=0$ as well. Hence $V_K^*=0$ on the entire base disk and on all of $K$. The elementary planar regularity facts and the product formula used here can be found in \cite{Ransford1995,Klimek1991}.
\end{proof}

\subsection{The weight and the derivative}

Take $\Phi$ from \cref{lem:smooth} and put $\Psi_\tau=\tau\Phi$ and $\psi_\tau=\Psi_\tau|_K$. The compact set \eqref{eq:global-K} contains the closed unit ball and $a=(0,3)$. Thus \cref{eq:base-identity,eq:radial-contact} give
\[
 V_{\tau,K,\psi_\tau}=F_\tau=\Psi_\tau,
 \qquad \Lambda(K,\psi_\tau)=[\tau,\infty).
\]
The next result proves failure of regularity at every larger parameter and gives the exact initial rate of increase.

\begin{proposition}\label{prop:global-bounds}
For $K$ in \eqref{eq:global-K}, $a=(0,3)$, and
\[
 \kappa_\tau=\frac{\tau\log(26/25)}{2\log5},
\]
the following bounds hold for every $h>0$:
\begin{equation}\label{eq:global-bounds}
 \min\{h,\kappa_\tau\}V_H^*(3)
 \le F_{\tau+h}(a)-\psi_\tau(a)
 \le hV_H^*(3).
\end{equation}
In particular,
\[
 F_{\tau+h}(a)=\psi_\tau(a)+hV_H^*(3)
 \qquad(0\le h\le\kappa_\tau),
\]
and
\[
 D_\tau(a)=V_H^*(3)
 \ge\frac{8\log3-4\log2}{9}>0=V_K^*(a).
\]
\end{proposition}
\begin{proof}
Fix $h>0$, write $\delta=\tau+h$, and put
\[
 \eta=\min\{h,\kappa_\tau\}.
\]
The function
\begin{equation}\label{eq:global-candidate}
 u_\eta(z,w)=\tau b+\tau\log|w|+\eta V_H^*(w)
\end{equation}
belongs to $\calL_{\tau+\eta}(\C^2)$. We first check its obstacle inequality off the base disk
$N=\overline\D\times\{3\}$.

On $A$ and on $\overline\D\times I_j$ we have $V_H^*(w)=0$. Thus \cref{lem:smooth} gives $u_\eta\le\Psi_\tau$ there. On $\Gamma$, where $|z|=1$ and $3\le w\le5$, we have
\begin{align*}
 \Psi_\tau(z,w)-\tau b-\tau\log w
 &=\frac{\tau}{2}\log(1+w^{-2})\\
 &\ge\frac{\tau}{2}\log(26/25)
 =\kappa_\tau\log5
 \ge\eta V_H^*(w).
\end{align*}
Hence $u_\eta\le\psi_\tau$ on $K\setminus N$.

Since $N$ is pluripolar and $\tau+\eta\le\delta$, the q.e. envelope identity \eqref{eq:qe-envelope} gives $F_\delta\ge u_\eta$ everywhere. Therefore
\[
 F_\delta(a)-\psi_\tau(a)\ge\eta V_H^*(3).
\]
This is the lower bound in \eqref{eq:global-bounds}. In particular, $a$ is irregular at every parameter $\delta>\tau$.

For the upper bound, take an arbitrary admissible $u\in\calL_\delta$, fix $|z|<1$, and set
\[
 r=(1-|z|^2)^{1/2}.
\]
If the restriction $u(z,\cdot)$ is identically $-\infty$, the required bound is immediate. Otherwise it is subharmonic. On $|w|\le r$, $(z,w)\in A$ and $u(z,w)\le\tau b$. For $w\in I\subset(2,3]$,
\[
 u(z,w)\le\tau b+\frac{\tau}{2}\log(|z|^2+w^2)
 \le\tau b+\tau\log(w/r),
\]
where the last inequality follows from
$(|z|^2+w^2)(1-|z|^2)\le w^2$.
Consequently, the function
\[
 v_z(w)=\begin{cases}
 0,&|w|\le r,\\
 \max\{u(z,w)-\tau b-\tau\log(|w|/r),0\},&|w|>r
 \end{cases}
\]
is subharmonic, belongs to $\calL_h(\C)$, and vanishes on $H_r$. The pasting at $|w|=r$ follows from upper semicontinuity and the bound on the disk. Hence $v_z\le hV_{H_r}$. For $|w|>r$, \cref{eq:global-shrunk} yields
\begin{align*}
 u(z,w)
 &\le\tau b+\tau\log(|w|/r)+hV_{H_r}(w)\\
 &\le\tau b+\tau\log|w|+hV_H^*(w)+\delta\log(1/r).
\end{align*}
This bound is independent of $u$. Take the supremum over all competitors and then the limit superior as $(z,w)\to(0,3)$. Since $r\to1$ and $V_H^*$ is upper semicontinuous, we obtain
\[
 F_\delta(a)\le\tau b+\tau\log3+hV_H^*(3)
 =\psi_\tau(a)+hV_H^*(3).
\]
This proves the upper bound in \eqref{eq:global-bounds}. The bounds coincide for $0<h\le\kappa_\tau$, giving the stated exact formula and derivative.
\end{proof}

\begin{proof}[Proof of \cref{thm:global-example}]
In dimension two, \cref{lem:global-geometry} proves all the geometric properties. The initial identity \eqref{eq:base-identity} gives regularity at $\tau$, while \cref{prop:global-bounds} gives failure at $a$ for every $\delta>\tau$. Hence the global regular parameter set is exactly $\{\tau\}$.

For $n>2$, use $K\times\cl\D^{\,n-2}$ and the point $(0,3,0,\ldots,0)$. Products preserve polynomial convexity and non-pluripolarity at every point. Non-plurithinness follows by restricting to the first two coordinates, and the product formula gives global pluriregularity. This product contains the Euclidean unit ball in $\C^n$, so \cref{eq:base-identity,eq:radial-contact} apply to the radial weight in dimension $n$. The candidate in \eqref{eq:global-candidate}, which depends only on $w$, still satisfies the obstacle inequality outside the pluripolar set $w=3$: adding coordinates can only increase $\Phi$. By \eqref{eq:qe-envelope}, it gives the same positive lower bound at $(0,3,0,\ldots,0)$. Dividing by $h$ and applying \cref{prop:sensitivity} gives $D_\tau(a)\ge V_H^*(3)$. This proves the theorem in every dimension.
\end{proof}

At $a=(0,3)$, the form $\ddc\Psi_\tau$ has rank $1$ and its kernel is the $w$-direction. Thus the positive derivative is consistent with \cref{thm:local-strict}. Also $a\in S$: the compact is non-plurithin at $a$, and \cref{prop:local-weights} shows that loss of weighted regularity rules out local pluriregularity there. This gives a point of $S$ on a globally pluriregular compact set.

\end{document}